\documentclass
[
11pt,
]
{amsart}
\usepackage{
amssymb
}
\newcommand{\no}[1]{#1}

\renewcommand{\no}[1]{}  \newcommand{\upDelta}{\Delta} 

\no{\usepackage{times}\usepackage[
subscriptcorrection, slantedGreek, 
nofontinfo]{mtpro}
\renewcommand{\Delta}{\upDelta}
}

\newtheorem{theorem}{Theorem}
\newtheorem{proposition}{Proposition}
\newtheorem{lemma}{Lemma}
\newtheorem{definition}{Definition}
\newtheorem{corollary}{Corollary}

\theoremstyle{remark}
\newtheorem{example}{Example}

\theoremstyle{remark}
\newtheorem{remark}{Remark}

\DeclareMathOperator{\supp}{supp}
\DeclareMathOperator{\Ker}{Ker}
\DeclareMathOperator{\Coker}{Coker}

\DeclareMathOperator{\Ran}{Ran}

\newcommand{\eps}{\varepsilon}
\newcommand{\R}{{\bf R}}
\newcommand{\Id}{\mbox{Id}}
\newcommand{\A}{\mathcal{A}}
\renewcommand{\H}{\mathcal{H}}
\newcommand{\B}{\mathcal{B}}

\renewcommand{\r}[1]{(\ref{#1})}
\newcommand{\PDO}{$\Psi$DO}
\newcommand{\be}[1]{\begin{equation}\label{#1}}
\newcommand{\ee}{\end{equation}}

\renewcommand{\d}{\mathrm{d}}
\renewcommand{\i}{\mathrm{i}}

\newcommand{\bo}{\partial M}

\newcommand{\Mint}{M^\text{\rm int}}

\title[Linearizing non-linear inverse problems]{Linearizing non-linear inverse problems\\ and an application to inverse backscattering}

\author[P. Stefanov]{Plamen Stefanov}
\address{Department of Mathematics, Purdue University, West Lafayette, IN 47907}
\thanks{First author partly supported by NSF Grant DMS-0800428}

\author[G. Uhlmann]{Gunther Uhlmann}
\address{Department of Mathematics, University of Washington, Seattle, WA 98195}
\thanks{Second author partly supported by NSF and a Walker Family Endowed Professorship}

\begin{document}

\begin{abstract} 
We propose an abstract approach to prove local uniqueness and conditional H\"older stability to non-linear inverse problems by linearization. The main condition is that, in addition to the injectivity of the linearization $A$, we need a stability estimate for $A$ as well. That condition is satisfied in particular, if $A^*A$ is an elliptic pseudo-differential operator. We apply this scheme to show  uniqueness and H\"older stability for the inverse backscattering problem for the acoustic equation near a constant sound speed.
\end{abstract}

\maketitle

\section{Introduction} 
Non-linear inverse problems are often linearized, and reduced to the problem of the injectivity of their linearization. If that linearization is an injective map with a closed range, then this implies local uniqueness and Lipshitz stability, see Theorem~\ref{thm_li} below. Most inverse problems in Mathematical Physics however are ill-posed and do not fall in this category. The mere fact that the linearization is injective, if it is, is not enough for local injectivity of the original problem when the closed range condition fails. The closed range condition is also equivalent to a linear stability estimate, see \r{13} below. 
The main purpose of this work is to propose a systematic approach to treat inverse problems for which the closed range condition may still hold if we replace the original spaces  with new ones satisfying  interpolation estimates, see \r{13}. A typical example is to replace $C^k$ or $H^s$ by $C^{k'}$, respectively $H^{s'}$, with different $k'$, $s'$. A sufficient condition for the linear stability estimate, besides the injectivity of the linearization $A$, is $A^*A$ to be an elliptic pseudo-differential operator (\PDO). The later is a standard consequence of the theory of elliptic \PDO s and elliptic \PDO s. Our main result in this direction is Theorem~\ref{thm1} that states, roughly speaking, that linearization plus an appropriate  stability estimate of the linearized map imply local injectivity and a conditional H\"older stability for the non-linear map. 

We recently used this approach for studying  boundary rigidity/lens rigidity questions for compact Riemannian manifolds with boundary \cite{SU-Duke, SU-JAMS, SU-lens}. 
Some of these ideas have been used before,  on a case by case basis,   for proving local uniqueness with or without a H\"older stability, see  for example \cite{Eskin,S-CPDE,SU-IBSP}.  
Our goal is to systematize this approach, and in particular to understand what makes (at least a large enough class of) ill-posed inverse problems ill posed, and in what cases one can get local uniqueness by linearization. In particular, we show why many severely ill-posed inverse problems cannot be treated this way --- the linearization is unstable in any pair of Sobolev spaces. The most famous example of this type is  Calder\'on's problem \cite{Calderon, SylvesterU87, U}.  

We give an application to the inverse backscattering problem for the acoustic equation. In this problem, we are trying to determine the sound speed or, equivalently, the index of refraction, of a medium by measuring the scattering amplitude with the direction of incidence of a plane wave opposite  to the direction of the reflected wave. In other words, we are measuring the echos produced by plane waves. This problem arises for instance in ultrasound tomography. We prove injectivity near constant sound speeds and we give a H\"older type conditional stability estimate, see Theorem~\ref{thm_e1}.  Conditional Lipshitz stability estimates in different norms were given in \cite{Wang98} based on the local uniqueness proof in \cite{SU-IBSP}.

In section~\ref{sec_IP} we prove the main result --- Theorem~\ref{thm1}. In section~\ref{sec_2}, we give sufficient conditions for the stability of the linearized inverse problem required by Theorem~\ref{thm1}. Those conditions are based on standard elliptic \PDO\ theory and Fredholm theory. 
Finally, in section~\ref{sec_4}, we give the application to the inverse backscattering problem.

\section{An abstract inverse problem}  \label{sec_IP}
We describe our abstract inverse problem now. 
Let $\A: \B_1\to \B_2$ be a continuous non-linear map between two Banach spaces. Consider the ``inverse problem'': 
\be{11}
\text{Given $h\in \mbox{Ran}(\A)$, find $f$ so that $\A(f)=h$.}
\ee
Here, we study only the (local) uniqueness and stability questions for \r{11}.

\begin{definition}  \label{def1}
We say that there is a weak local uniqueness for \r{11} near $f_0\in \B_1$, if  there exists a neighborhood of $f_0$ so that for any $f$ in that neighborhood that solves 
\be{13a}
\A(f)=\A(f_0),
\ee
we have $f=f_0$.
\end{definition}

\begin{definition}  \label{def2}
We say that there is a strong local uniqueness for \r{11} near $f_0\in \B_2$, if  there exists a neighborhood of $f_0$ so that for any $f_1$, $f_2$ in that neighborhood with
\be{14}
\A(f_1)=\A(f_2),
\ee
we have $f_1=f_2$.
\end{definition}

Besides uniqueness, we are often interested in stability estimates. By a classical argument, if $\A: U \to V$ is injective, where $U$ and $V$ are open, and $K\subset U$ is compact, 
then $\A^{-1}$, restricted to $\A(K)$, is continuous. The compactness assumption is one of the ways ill-posedness of many (but not all) inverse problems manifests itself. The continuity of $\A^{-1}$ us viewed as stability, and such kind of results are called conditional stability. One of the central problems in Inverse Problems is to estimate the modulus of continuity of $\A^{-1}$ (restricted appropriately). In other words, we want to find a function $\phi(t)$, $t\to 0$, so that $\phi\to0$, as $t\to0$, and 
\[
\|f_1-f_2\|_{\B_1}   \le \phi\left(\|\A(f_1) -\A(f_2)\|_{\B_2}\right)
\]
for all $f_1$, $f_2$ in some neighborhood of a fixed $f_0$, and possibly also, restricted to a compact subset $K$ of $\B_1$ (hence the stability is conditional). The estimate is of H\"older type, if one can choose $\phi(t) = Ct^\alpha$, $0<\alpha\le1$, and of Lipschitz type, if $\phi(t) = Ct$. We view in this paper inverse problems that allow H\"older estimates as ``stable'' (Lipschitz stability occurs rarely); while any other are viewed as ``unstable''. Recall that for the EIT problem, $\phi(t) = C\left(\log(1/t)\right)^{-\mu}$, $\mu>0$, that tends to $0$, as $t\to0$, much slower than $t^\alpha$. 

The choice of the compact set $K$ also matters. If it is ``too small'', the conditional stability is trivial and not interesting. For example, if $K$ is finite dimensional, and $\A$ is smooth enough, one always has a Lipshitz stability estimate under the assumption that $\A$ has an injective differential, see below. If $\B_1$ is $H^s(M)$ or $C^k(M)$ with a compact set (or manifold) $M$, we are interested whether one can choose $K$ to be $H^{s_1}(M)$ or $C^{k_1}(M)$ with $s_1>s$, or $k_1>k$. A subspace $K$  consisting of real analytic functions only, for example, will be considered ``too small''.

\subsection{``Inverse Problems'' in $\R^n$} 
We start with the trivial case when $\B_1=\R^n$, $\B_2=\R^m$. Then the problem is reduced to solving \r{11} with $\A :\R^n\to\R^m$. Assume that $\A\in C^2$ ($\A\in C^1$  actually suffices, see Theorem~\ref{thm_li} below). Let $A_{x_0}$ be the differential of $\A$ at $x=x_0$. Then
\be{1a1}
\A(x) = \A(x_0) + A_{x_0}(x-x_0) + R_{x_0}(x)  \quad \mbox{with}\;  |R_{x_0}(x)| \le C_{x_0}|x-x_0|^2,
\ee
for $x$ near $x_0$. Assume now that $A_{x_0}$ is injective (which can happen only when $m\ge n$). This immediately implies the estimate
\be{1a2}
|h| \le   C_0|A_{x_0}h|, \quad \forall h\in \R^n.
\ee
The easiest way to prove \r{1a2} is to note that one can choose $1/C_0$ to be the minimum of $|Ah|$ for $h$ on the (compact) unit sphere. 
Then the local uniqueness question about the number of possible solutions of $\A(x) = \A(x_0)$ near $x=x_0$ can be answered as follows. We have
\[
\begin{split}
|x-x_0|/C_0  \le 
|A_{x_0}(x-x_0)| &\le  |\A(x) - \A(x_0)|+ |R_{x_0}(x)| \\
&\le  |\A(x) - \A(x_0)| + C_{x_0} |x-x_0|^2.
\end{split}
\]
Therefore, if $|x-x_0|\le (2C_0C_{x_0})^{-1}$, one gets
\be{1a3}
|x-x_0|\le 2C_0|\A(x) - \A(x_0)|.
\ee
We therefore have local uniqueness and a Lipshitz stability estimate. Moreover, one can choose $C_{x_0}$ to be locally independent of $x_0$, if $\A\in C^2$, and the same applies to $C_0$. Thus one gets that there exists a neighborhood $U$ of $x_0$ in $\R^n$, so that for any $x_1\in U$, $x_2\in U$, the equality $\A(x_1) = \A(x_2)$ implies $x_1=x_2$. Moreover, similarly to \r{1a3}, one has
\be{1a4}
|x_1-x_2|\le 2C_0|\A(x_1) - \A(x_2)|, \quad x_1\in U, \; x_2\in U.
\ee

Note that $\B_2$ can be infinite dimensional here, and the arguments do not change. In particular, if $\A\in C^{1,\mu}$, $\mu>0$, and $A_{x_0}$ is injective, then one always has local uniqueness and the  stability estimate \r{1a4} of Lipshitz type. If $\B_1$ is infinite dimensional, then $\A|_K$ has that property for any finite dimensional $K$.

\subsection{A well posed inverse problem. The local injectivity theorem in Banach spaces. } We return to the case where $\B_{1,2}$ are finitely dimensional Banach spaces. 
Consider now a truly multidimensional version of our abstract inverse problem under assumptions that make it well-posed, and therefore, not typical. The main assumption is that $A_{x_0} :\B_1\to A_{x_0}(\B_1)$ is an isomorphism, i.e., that  the differential $A_{x_0}$ is not only injective but also has a closed range. Then we have the following theorem that is closely related to the inverse function theorem and the implicit function theorems in Banach spaces ($A_{x_0}$ is assumed to be an isomorphism from $\B_1$ to $\B_2$ then). 

\begin{theorem}[Local Injectivity Theorem, \cite{AbrahamMR}]  \label{thm_li} 
Let $\A : U\subset \B_1\to B_2$ be $C^1$, with $U\ni f_0$ open, let $A_{f_0}$ be injective and let it have a closed range. Then there exists a neighborhood $V\subset U$ of $f_0$, on which $\A$ is injective. Moreover, the inverse $\A^{-1}: \A(V)\to U$ is Lipshitz continuous. 
\end{theorem}

This theorem shows that we have the same conclusions as in the finitely dimensional case, with a Lipshitz stability of non-conditional type. 

We refer to \cite{AbrahamMR}] for the definition of the differential $A_{f_0}$ of $\A$ at a fixed $f_0$ and  for a definition of  (a Fr\'echet) differentiable map. In particular, if the G\^{a}teaux derivative (the directional derivative in all directions) exists in some open set and it is continuous there, then $\A$ is Fr\'echet differentiable as well and the two derivatives coincide. Then $\A$ is said to be $C^1$. 

\subsection{An abstract ill-posed inverse problem} We turn out attention now to an abstract  inverse problem that may not be well posed.  
In order to study the inverse problem \r{11} by linearization, we assume that the differential $A_{f_0}$ of $\A$ exists at $f=f_0$ with some $f_0$. Actually, we will assume something slightly stronger --- a quadratic estimate on the remainder, that roughly speaking means that $\A$ is $C^{1,1}$. Namely, we assume that 
\be{12}
\A(f) = \A(f_0) + A_{f_0}(f-f_0) + R_{f_0}(f), \quad \mbox{with}\;  \|R_{f_0}(f)\|_{\B_2}\le C_{f_0}\|f-f_0\|_{\B_1}^2,
\ee
for $f$ in some neighborhood in $\B_1$ of some $f_0\in \B_1$.

It may happen that we have the weaker estimate
\be{12w}
\|R_{f_0}(f)\|_{\tilde\B_2}\le C_{f_0}\|f-f_0\|_{\tilde\B_1}^2
\ee
for $f$ in some neighborhood in $\tilde\B_1$ of some $f_0\in \tilde\B_1$, where $\tilde\B_1\subset \B_1$, $\tilde \B_2\supset \B_2$. 
Here and below, $\tilde\B_1\subset \B_1$, $\tilde \B_2\supset \B_2$  mean also that 
\be{12a}
\|\cdot\|_{\B_1}\le C\|\cdot\|_{\tilde\B_1},\quad \|\cdot\|_{\tilde\B_2}\le C\|\cdot\|_{\B_2}.
\ee
In such case, we can simply replace $\B_2$ by $\tilde\B_2$; and $\B_1$ by $\tilde \B_1$. In the boundary rigidity problem \cite{SU-JAMS}, actually $\tilde \B_1=C^1(M)\supset \B_1=C^2(M)$, i.e., \r{12} is even stronger. Keeping  $\tilde\B_1$, $\tilde\B_2$ and assuming appropriate interpolation estimates would improve a bit the exponent $\mu_1\mu_2$ in the H\"older exponent in Theorem~\ref{thm1} below but for the sake of simplicity, we will not pursue this. In other words, we have some freedom how to chose our spaces $\B_{1,2}$, and we choose them in a way that makes \r{12} true. This choice might not be the optimal for the analysis of $A_{f_0}$ however.

Assume now that $A_{f_0}$ is injective. This does not imply automatically a stability estimate of the type \r{1a2}, of course. Assume in addition, that such an estimate holds in some norms, i.e.,  
\be{13}
\|h\|_{\B_1'}\le C\|A_{f_0}h\|_{\B_2'}, \quad \forall h\in \B_1,
\ee
with some Banach spaces $\B_2'\subset \B_2$, $\B_1'\supset\B_1$.  In applications, we want not only such $\B_1'$, $\B_2'$ to exist (they always do, one can set $\B_1'=\B_1$,  $\|A_{f_0}h\|_{\B_2'} = \|h \|_{\B_1}$ on $\Ran A_{f_0}$) but we also want they to be ``reasonable'' spaces independent of $f_0$, typically some  $H^s$ or $C^k$ spaces. 
Then, if $\B_1=\B_1'$, $\B_2=\B_2'$ one can apply Theorem~\ref{thm_li} to 
  prove local injectivity and a Lipshitz stability estimate of the type \r{1a4}. The inverse problem is well-posed then. 

If the pairs of spaces in \r{12} and \r{13} cannot be chosen to be the same, and often that is the case, one can still prove local uniqueness and a stability estimate but of a H\"older and conditional type, if certain interpolation estimates are satisfied, as we show below. Then the inverse problem is ill-posed but we think of it as mildly ill-posed. 

Before formulating this as a formal statement, we will give an example where \r{13} does not hold (in any Sobolev spaces), and there is no local uniqueness.

\begin{example}   \label{ex1}
Let $l^2$ be the Hilbert space of sequences $x=\{x_k\}_{k=1}^\infty$ with norm $\|x\|^2=\sum |x_k|^2$. The Sobolev spaces $h^s$ are defined through the norms $\|x\|^2_{h^s} = \sum k^{2s}|x_k|^2$. We think of the components of $x$ as the Fourier coefficients of a $2\pi$-periodic function, which explains why we call $h^s$ a Sobolev space. 
Define the non-linear map $\A$ by
\be{15}
\A (x) = Ex - (x,a)x,
\ee
where 
\[
a = \{1/k\}, \quad E = \textrm{diag}\{e^{-k}\}.
\]
Clearly, $\A(0)=0$, the linearization of $\A$ near $x=0$ is $A_0= E$, and the latter is an injective map. Moreover, there is an inverse $E^{-1}$ with a dense domain, but $E^{-1}$ is unbounded as an operator from any Sobolev space to any other one, thus \r{13} does not hold in such spaces. 
Since $[\A (x)]_k = e^{-k}x_k - x_k\sum x_m/m$, we get that
\[
x^{(k)} = (0,\dots ke^{-k},0,\dots),
\] 
where the non-zero entry is the $k$-th one, is a solution to $\A(x)=0$. Therefore, in any neighborhood of $0$ in $l^2$ there are infinitely many solutions of $\A(x)=0$ despite the fact that the linearization $A_0$ is injective. Moreover, for any $s$, in any neighborhood of $0$ in $h^s$, there are still infinitely many solutions. Therefore, there is no local uniqueness in this case, neither weak nor strong, in any Sobolev space. One can still get local uniqueness for $\|x\|_*\ll1$ by choosing an  $h^\infty$ type of norm $\|\cdot\|_*$ with an exponential weight, namely $\|x\|_* = \|E^{-1}x\|$. This however translates into a neighborhood of the origin of $2\pi$-periodic functions that involves certain real analytic functions only. We think of such topology as ``unreasonably restrictive''. 

\medskip
Based on that example, one can consider  the following map in $L^2(\R^n)$
\[
\A(f) = \phi*f - (f,a)f.
\]
Choose  $\hat\phi(\xi) = \sum \chi_{\{k-1\le |\xi|\le k\}}(\xi)e^{-k|\xi|}$, where $\chi_K$ stands for the characteristic function of $K$. Also, fix   $a\in L^2$ with $|\hat a(\xi)|\ge (1+|\xi|)^{-m}$ for some $m$. Then $f=0$ solves $\A(f)=0$, and $\A$ has an injective differential at $f=0$.  Also, if $\hat f$ is supported in $\{k-1\le |\xi|\le k\}$, and $\hat f$ solves $(\hat f,\hat a) =(2\pi)^{n} e^{-k|\xi|}$, such an $f$ would also solve $\A(f)=0$. This gives us a sequence of such solutions, converging exponentially fast to $0$ in any $H^s$. 
\end{example}  

Our main abstract theorem for linearizing inverse problems with a ``stable" linearization is the following.

\begin{theorem}  \label{thm1}
Let $\A$ be as above, and assume \r{12}, \r{13} with $\B_1\subset \B_1'$, $\B_2' \subset \B_2$ as above. 
 Assume also that 
there exist Banach spaces $\B_2''\subset \B_2'$, $\B_1''\subset \B_1$ so that  $A_{f_0}$ continuously maps $\B_1''$ into $\B_2''$ and 
the following interpolation estimates hold 
\be{int}
\|u\|_{\B_2'}\le C\|u\|_{\B_2}^{\mu_2} \|u\|^{1-\mu_2}_{\B_2''}, \quad 
\|h\|_{\B_1}\le C\|h\|_{\B_1'}^{\mu_1} \|h\|^{1-\mu_1}_{\B_1''}
\quad  \mu_1,\,\mu_2 \in(0,1],  \quad \mu_1\mu_2>1/2.
\ee

(a) For any $K>0$   there exists $\epsilon>0$, so that  for any $f$ with
\be{16}
\|f-f_0\|_{\B_1} \le\epsilon, \quad \|f\|_{\B_1''}\le K, 
\ee
one has the conditional stability estimate
\be{stab}
\|f-f_0\|_{\B_1} \le C(K) \|\A(f) - \A(f_0)\|^{\mu_1\mu_2}_{\B_2}, \quad C(K) = CK^{2-\mu_1-\mu_2} .
\ee
In particular, there is a weak local uniqueness near $f_0$, i.e., if $\A(f)=\A(f_0)$, then $f=f_0$.

(b) Assume in addition that there is a Banach space $\mathcal{K}\subset \B_1''$ so that \r{13} holds for $f_0$ replaced with $f$ close enough to $f_0$ in $\mathcal{K}$, and $A_{f}: B_1''\to B_2''$ is uniformly bounded for such $f$. Then 
there exists $\epsilon>0$, so that  for any $f_1$, $f_2$ with 
\be{sm}
\|f_1-f_0\|_{\mathcal{K}}\le\epsilon, \quad  \|f_2-f_0\|_{\mathcal{K}} \le\epsilon,
\ee
one has the conditional stability estimate
\be{stab2}
\|f_1-f_2\|_{\B_1} \le C  \|\A(f_1) - \A(f_2)\|^{\mu_1\mu_2}_{\B_2}.
\ee
In particular, there is a strong local uniqueness near $f_0$, i.e., if $\A(f_1)=\A(f_2)$, then $f_1=f_2$.
\end{theorem}

\begin{remark}
Typical choices of the Banach spaces above are $C^k$ or $H^s$ spaces. Interpolation estimates for those norms are well known.
In $\R^n$, for example, one has
\be{inter}
\|u\|_{H^s} \le C \|u\|_{H^{s_1}}^{\alpha_1} \|u\|_{H^{s_2}}^{\alpha_2} , \quad s= s_1\alpha_1+s_2\alpha_2, \quad \alpha_1+\alpha_2=1,\quad \alpha_1\ge0, \; \alpha_2\ge0.
\ee
This can be easily verified by using the Fourier transform. Similar estimates hold for Sobolev spaces and the $C^k$ spaces in domains/manifolds, see \cite{Triebel}. For example, if $M$ is a compact manifold with boundary, then
\be{inter2}
\|u\|_{C^k(M)} \le C \|u\|_{C^{k_1}(M)}^{\alpha_1} \|u\|_{C^{k_2}(M)}^{\alpha_2} , \quad k= k_1\alpha_1+k_2\alpha_2, \quad \alpha_1+\alpha_2=1,\quad \alpha_1\ge0, \; \alpha_2\ge0.
\ee
Interpolation estimates that include an $H^s$ and a $C^k$ norm can be obtained by using Sobolev embedding theorems. 
\end{remark}

\begin{remark}
Assume, for example, that all spaces above are Sobolev spaces. Then one can make the exponent $\mu_1\mu_2$ in \r{stab}, \r{stab2} arbitrary close  to $1$, thus making the stability estimate ``almost Lipshitz''. The price to pay for that is to assume that $f$ is a priori bounded in $H^s$ with $s\gg1$, see \r{inter}. Alternatively, one may try to choose $\B_1''$, $\mathcal{K}$  close to $\B_1$, which is the natural limit but that will decrease $\mu_1$. Since $\mu_1>1/(2\mu_2)\ge 1/2$, this puts a limit on how close $\B_1''$, $\mathcal{K}$ can be to $\B_1$. If $\B_1=H^{s}$, $\B_1'=H^{s'}$, for example, $s'<s$, 
then we will have at least $\B_1''=H^{s''}$ with $s''$ so $(s'+s'')/2>s$. Therefore, we cannot take $s''>s$ arbitrary close to $s$ unless $s'=s$. The choice of $s''$ in this case is further restricted by $\mu_2$. 
\end{remark}

\begin{remark}
The second inequality in \r{16} is a typical compactness condition, when we work with Sobolev spaces or $C^k$ spaces in a  bounded domain or on a compact manifold. It may seem strange that there is no such explicit condition in (b). It is actually there, since \r{sm} implies such a condition with $\B_2''$ replaces by $\mathcal{K}$. Actually, using interpolation estimates, one could write \r{16} in the form \r{sm}, and vice
-versa, with slight change of the Banach spaces (i.e., small change of $k$ in $C^k$, etc.). We formulated \r{sm} in a form different than \r{16} in order to avoid introducing yet another Banach space. What makes \r{stab} and \r{stab2} conditional estimates is that in the conditions \r{16}, respectively \r{sm}, there is at least one norm stronger that the norm in the l.h.s.\ of \r{stab}, respectively \r{stab2}. Note also that \r{stab}, \r{stab2} can be formulated with different choices of the norms.
\end{remark}

\begin{proof}[Proof of Theorem~\ref{thm1}]
 We start with (a). 
By \r{12} and \r{12a},
\[
\|A_{f_0}(f-f_0)\|_{\B_2}  \le \|\A(f)-\A(f_0)\|_{\B_2} + C\|f-f_0\|^2_{\B_1}
\]
By \r{13} and the H\"older inequality $(a+b)^{\mu} \le a^{\mu} + b^{\mu}$ for $a\ge0$, $b\ge0$, we get
\[
\|A_{f_0}(f-f_0)\|_{\B_2'}  \le C K^{1-\mu_2}\left( \|\A(f)-\A(f_0)\|_{\B_2}^{\mu_2} + \|f-f_0\|^{2\mu_2}_{\B_1}\right).
\]
The constant $C K^{1-\mu_2}$ comes from estimating the term $\|A_{f_0}(f-f_0)\|_{\B_2''}^{1-\mu_2}$ that is bounded by the assumption on $A_{f_0}$ and by the second inequality in \r{16}.
By \r{13}, and the second interpolation inequality in \r{int},
\[
\|f-f_0\|_{\B_1}\le  C K^{2-\mu_1-\mu_2} \left( \|\A(f)-\A(f_0)\|_{\B_2}^{\mu_1\mu_2} + \|f-f_0\|^{2\mu_1\mu_2}_{\B_1}\right).
\]
Since $2\mu_1\mu_2>1$, for $\epsilon\ll1$ we get \r{stab}. Note that the condition on $\epsilon$ has the form $\epsilon\le (CK^{2-\mu_1-\mu_2})^{-1}$, where $C$ depends on $\A$ only. This proves (a).

To prove (b) we note that the same proof can be applied under the assumptions of (b). Indeed, we start with 
\[
\|A_{f_1}(f_2-f_1)\|_{\B_2}  \le \|\A(f_2)-\A(f_1)\|_{\B_2} + C\|f_2-f_1\|^2_{\B_1}.
\]
Then we proceed as above under the condition $\|f_1-f_0\|_{\mathcal{K}}\le\epsilon_1\ll1$ that would guarantee \r{13} with $f_0$ replaced by $f_1$. Then  \r{stab2} holds if   $\|f_2-f_1\|_{\B_1}\le\epsilon_2\ll1$. Those two conditions certainly hold if $\epsilon$ in \r{sm} is small enough.
\end{proof}

\begin{remark} The variety of norms in the theorem above is typical. For example, in the boundary rigidity problem, we first apply $A_{f_0}^*$ first (called there $I^*$, and  the adjoint operator  is taken in $L^2$ spaces  with respect to suitable measures). Then $A_{f_0}$ is replaced by $A_{f_0}^*A_{f_0}$, and 
we made the following choices, see \cite{SU-AJM}:
\begin{alignat}{5}
\B_1& = C^2(M), \quad & \B_1'&= L^2(M),
\quad   
&\B_1''&= H^s(M),  \; s\gg1,&\quad \mathcal{K}=C^k(M), \; k\gg s,
\\
\B_2& = L^\infty(M_1),\quad  &\B_2'&= \tilde H^2(M_1), \quad &\B_2'''&= H^{s-1}(M_1),
\end{alignat}
where $\Mint_1\Supset M$, and $\tilde H^2(M_1)$ is an appropriate Hilbert space such that $H^2(M_1)\subset \tilde H^2(M_1) \subset H^1(M_1)$. There are other complications there, coming from the fact that we have to work with equivalence classes of isometric metrics. Then $I^*I$ acting on symmetric tensors in $H^s(M)$, extended as zero to $M_1$, is not a bounded operator with values in $H^{s-1}(M_1)$ even though $I^*I$ is a \PDO\ of order $-1$. The reason is that $f\in H^s(M)$, extended as zero, belongs to $H^s(M_1)$ only if the first $s$ normal derivatives of $f$ (assuming that $s$ is integer) vanish on $\bo$. This is the reason for introducing the space $\tilde H^2$. 
A stable recovery of the derivatives of the metric on $\bo$ from the boundary distance function is used in \cite{SU-AJM} to deal with this problem. 
\end{remark}

\section{Stability of linear inverse problems} \label{sec_2}
Most of the material in this section is based on the theory of elliptic \PDO s and on the  theory of Fredholm operators. We give proofs for the sake of completeness of the exposition. 

Let $A: \B_1 \to \B_2$ be a bounded linear map between two Banach spaces. Assume in addition, that $A$ is injective. We are interested whether one has the stability estimate
\be{21}
\|f\|_{\B_1}  \le C\|Af\|_{\B_2}  
\ee
with some constant $C>0$. It is well known, that if $A$ is smoothing, and $\B_1$ and $\B_2$ are among $C^k$ or $H^s$ with finite $k$, or $s$, then there cannot be such a stability estimate. We will quickly review those arguments below. 

First, note that \r{21} is not equivalent to invertibility, because (the  closure of) the range $\Ran A$ of $A$ can be much smaller than $\B_2$. One the other hand, one can replace $\B_2$ by $B_2'=\overline{\Ran A}$, and then \r{21} is equivalent to the invertibility of $A :\B_1\to \B_2'$, as the next proposition shows.

\begin{proposition}  \label{pr21}
\r{21} holds if and only if $A$ is injective and $\Ran A$ is closed.
\end{proposition}

\begin{proof}
Assume that $\textrm{\rm Ran}\, A$ is closed and $A$ is injective. Then $A :\B_1\to \B_2'$ has range that coincides with $\B_2'$. Since $A$ is injective, then $A^{-1}$ exists with domain the whole $\B_2'$. By the open mapping theorem, $A^{-1}$ is bounded, hence \r{21} holds. 

Assume \r{21}. Then the injectivity and the closedness of $\Ran A$ follow  easily. 
\end{proof}

Proving that $\Ran A$ is closed is not always straightforward. Proving that it is not (when it is not), is usually easier. 
As an example, suppose that $\B_1=\B_2=L^2(\Omega)$, where $\Omega$ is a domain in $\R^n$ with smooth boundary, or $\Omega=\R^n$. If there exists $s>0$ so that $Af\in H^s(\Omega)$ for any $f\in L^2(\Omega)$, then \r{21} cannot be true. Indeed, if we know that  $A$ has a trivial cokernel (for example, because it is injective and self-adjoint, as the operator $N_w$ below), then $\Ran A$ cannot be closed. In the general case, assume that $\Ran A$ is closed. Then the latter is a Banach space as well, and $A$ is invertible on it. But  $A$ is compact, so we get that the identity on $L^2(\Omega)$ is compact, that is not true.

On the other hand, \r{21} may still hold with $\B_2=H^s(\Omega)$, if $\Ran A$ is closed in $H^s(\Omega)$. Take for example $A=(1-\Delta)^{-1}$ on $L^2(\R^n)$, and $s=2$. 
Then one can ask the more general question: is there a stability estimate of the type
\be{22}
\|f\|_{L^2} \le C\|Af\|_{H^s} 
\ee
for some $s>0$? If so,  it will be a weaker substitute for \r{21}, if the latter does not hold but it will still allow us to apply Theorem~\ref{thm1}. For \r{22} to be true, $\A$ must map $L^2$ into $H^s$, therefore, this is the same question as before, with $\B_2=H^s$.

\begin{example}[the weighted X-ray transform] 
\label{ex2}
In $\R^n$, let $I_w$ be the weighted X-ray transform
\be{23}
I_wf(x,\theta) = \int_{\R} w(x,\theta) f(x+t\theta)\, \d t,\quad x\in \R^n, \; \theta\in S^{n-1},
\ee
where $w$ is a continuous function. Let us restrict $I_w$ to functions supported in $\bar\Omega$, where $\Omega\subset \R^n$ is a strictly convex domain  with a smooth boundary. One can assume then that $(x,\theta)\in \partial_- S\Omega$, which means that $x\in\partial\Omega$, and $\theta\in S^{n-1}$ points into $\Omega$. We equip $\partial_- S\Omega$ with the measure $\d\Sigma = |\theta\cdot\nu(x)|\, \d S_x\,\d \theta $, where $\d S_x$ is the measure on $\partial\Omega$ induced by the Euclidean measure in $\R^n$, and $\d\theta$ is the measure on $S^{n-1}$. Then one can show that $I_w :L^2(\Omega) \to L^2(\partial_- S\Omega, \, \d\Sigma)$ \cite{SH-book}, see also the remark following \r{26} below. Let
\be{24}
N_w = I_w^*I_w : L^2(\Omega) \to L^2(\Omega)
\ee
be the ``normal operator''. We can think of the lines through $\Omega$ as a manifold, and $\partial_- S\Omega$ then is a global chart for them (singular at the lines tangent to $\partial\Omega$). 
The measure $\d\Sigma$ remains invariant if we replace $\partial\Omega$ by another strictly convex surface encompassing $\Omega$. Thus one can expect that $N$ will not depend on $\partial\Omega$. Indeed, on can easily get that (see \cite{SU-Duke})
\be{25}
N_w f(x) = c_n \int \frac{W(x,y)f(y)}{|x-y|^{n-1}}   \,\d y
\ee
with 
\be{26}
W(x,y) = \bar w\bigg(x,-\frac{x-y}{|x-y|}\bigg) w\bigg(y,-\frac{x-y}{|x-y|}\bigg) 
+\bar w\bigg(x,\frac{x-y}{|x-y|}\bigg) w\bigg(y,\frac{x-y}{|x-y|}\bigg).
\ee
Now we think of $N_w$ as the operator defined by the formula above, with $w(x,\theta)$ extended for all $x$. It is easy to see that $N_w$ maps $L^2_{\rm comp}(\R^n)$ into $L^2_{\rm loc}(\R^n)$. In particular, $N_w$ is bounded on $L^2(\Omega)$. That shows that $I :L^2(\Omega) \to L^2(\partial_- S\Omega, \, \d\Sigma)$ is bounded. 

Let $\Omega_1\Supset\Omega$ be another domain with the properties of $\Omega$. Then $N : L^2(\Omega)\to L^2(\Omega_1)$ is also bounded, and $I$ is injective if and only if $N : L^2(\Omega)\to L^2(\Omega_1)$ is injective, where one can replace $\Omega_1$ by $\Omega$ as well, see Lemma~\ref{lemma1} below. 

If $w\in C^\infty(\bar\Omega)$, then $N$ is a \PDO\ of order $-1$, as we will see below. Thus $N$ maps $L^2(\Omega)$ into $H^1(\Omega)$ and a stability estimate for $N$ of the kind \r{22}, with $\B_1=\B_2=L^2(\Omega)$, or $\B_2=L^2(\Omega_1)$, is not possible. If $w\in C^1(\bar\Omega)$ only, then one can still show that  $N$ maps $L^2(\Omega)$ into $H^1(\Omega)$ using the theory of operators with singular kernels, see \cite{FSU} for details. We then change $\B_2$ and replace it with $H^1(M_1)$. The question we ask now is the following: assuming that $I$ (and therefore $N$) is injective for some class of $w$'s, what conditions on $w$ would guarantee
the stability estimate
\be{27}
\|f\|_{L^2(\Omega)} \le C\|N_w f\|_{H^1(\Omega_1)}, \quad \forall f\in L^2(\Omega).
\ee
We will show below that a sufficient condition for that is that $N$ is elliptic, that can be easily formulated in terms of $w$. 
\end{example}

We return to the general case and the estimate \r{21}. 
In the example above, if $N_w$ is elliptic, then it is also a Fredholm operator. Recall that $F: \B_1\to\B_2$ is called Fredholm, if there exist bounded operators $A: \B_1\to\B_2$ and $B: \B_2\to\B_1$ so that $AF-\Id$ and $FB-\Id$ are compact. Then $\Ran F$ is necessarily closed, and $F$ has a finitely dimensional kernel $\Ker F$ and a finitely dimensional cokernel $\Coker A=\B_2/\Ran A$ . Equivalently, $F$ is Fredholm if its kernel and cokernel are finitely dimensional.

In case the inverse problem is over-determined, we cannot expect that $A$ would have a finite cokernel, as in the example above. For this reason, if $\B_{2}$ is a  Hilbert space, we will study whether $N=A^*A$ is a Fredholm operator, 
and the consequences of that. In the applied literature, applying $A^*$ is viewed a ``back-projection'' that, in inverse boundary value problems,  returns our data form the boundary back to the domain. Another reason to study $A^*A$ instead of $A$   is that, as in the example above, $A$ might be an FIO, but $A^*A$ is a \PDO, thus easier to study. The latter is not alway true for any FIO, so some conditions are needed. In case of integrals along geodesics, this is guaranteed by the no-conjugate points assumption, that is certainly true for the Euclidean metric. Again the example  above shows that we may need to restrict $N$ to a subspace $\mathcal{L}$ of $\H$ in order to avoid working with \PDO s on manifolds with boundary. Then $F|_{\mathcal{L}}$ may not have finite cokernel anymore. For this reason, we will only study operators $F$ that are \textit{upper semi-Fredholm}, i.e., $\Ker F$ is finite dimensional, and $\Ran F$ is closed. Equivalently, $F$ is  upper semi-Fredholm, if there exits a bounded left inverse $A$ modulo compact operators ($AF-\Id$ is compact). Note that if $N=A^*A$ is Fredholm, then $A$ itself is upper semi-Fredholm.

We are not losing generality when  replacing $A$ by $N=A^*A$ as the next lemma shows. Note that also stability estimates for $N$ of the type \r{27} and \r{29} below imply stability estimates for $A$ of the type \r{21}, \r{22} as well. 

\begin{lemma}  \label{lemma1}
Let $A : \B_1\to \B_2$ be bounded, where $\B_{1,2}$ are Hilbert spaces, and let $\mathcal{L}\subset \B_1$ be a (closed) subspace. 
The following statements are equivalent.

(a) $A: \mathcal{L} \to \B_2$ is injective;

(b) $A^*A : \mathcal{L}\to \B_1$ is injective. 

(c) $\Pi_{\mathcal{L}}A^*A : \mathcal{L}\to \mathcal{L}$ is injective, where $\Pi_{\mathcal{L}}$ is the orthogonal projection to $\mathcal{L}$.  
\end{lemma}

\begin{proof}
Let $f\in \mathcal{L}$. Then $(A^*Af,f) = \|Af\|^2$. This proves the implication (a) $\Rightarrow$ (b). If $Af=0$ with $f\in \mathcal{L}$, then $A^*Af=0$ as well,  therefore  (b) $\Rightarrow$ (a). The equivalence of (c) and (a) follows in the same way.
\end{proof}

\begin{proposition}  \label{pr2F}
Let $N :\B_1\to \B_2$ be upper semi-Fredholm. 
Assume in addition, that $N$ is injective on some (closed) subspace $\mathcal{L}\subset \H$. Then:

(a) There exists $C>0$, so that
\be{29}
\|f\|_{\B_2} \le C \|Nf\|_{\B_1}, \quad \forall f\in \mathcal{L}.
\ee

(b) Estimate \r{29} remains true for any other bounded operator close enough to $N$ with a uniform $C>0$.
\end{proposition}

Clearly, one can always take $\mathcal{L}=(\Ker N)^\perp$ if $\B_1$ is a Hilbert space. The way we formulated the proposition above however is intended to underline the fact  that if we can prove injectivity on some explicitly given subspace that is of interest to us ($L^2(\Omega)$ in the example above), then we automatically have stability.

\begin{proof}[Proof of Proposition~\ref{pr2F}]  
The operator $F : \mathcal{L}\to F(\mathcal{L})$ is bounded, injective,  and has a closed range. In particular, it is also upper semi-Fredholm. By Proposition~\ref{pr21}, \r{29} holds. This proves (b). 

To prove (b), it is enough to write
\be{210a}
\|f\|_{\B_1}/C \le \|N_{q_0}f\|_{\B_2}\le  \|(N_q- N_{q_0})f\|_{\B_2}+ \|N_{q}f\|_{\B_2}.
\ee
\end{proof}

\begin{remark}
In the tensor tomography problems, one has to use a finer version of Proposition~\ref{pr2F}(b). Then $N=N_g$ depends continuously on the metric $g$, and $N_g : H^1(M) \to \tilde H^2(M_1)$, where $M_1\Supset M$ and tensors in $M$ are alway extended as $0$ outside $M$.  The space $\tilde H^2(M_1)$ has the property that  $H^2(M_1)\subset \tilde H^2(M_1) \subset  H^1(M_1)$. 
On the other hand, $\tilde H^2(M_1)$ is ``too large'', and $\Ran N_g$ is not closed there. 
It turns out, that one can construct a left parametrix $Q_g$ so that $Q_gN_g=\Id+K_g$, where $K_g$ is a smoothing operator. However, $Q_g :\tilde H^2(M_1) \to H^1(M)$ is not bounded; we can show only that $Q_g :\tilde H^2(M_1) \to L^2(M)$ is bounded. On the other hand, $Q_gN_g$ is bounded on $H^1(M)$, of course. This requires some modification of the arguments above,
and the resulting estimate is
\[
\|f^s\|_{L^2(M)} \le C\|N_gf\|_{\tilde H^2(M)}, \quad \forall f\in H^1(M)
\]
with $C>0$ locally uniform in $g$. 
Here $f^s$ is the orthogonal projection of $f$ onto a certain space, called solenoidal tensors,  where $N_g$ is injective by assumption. 
Note that  \r{210a} does not apply with $\B_1=L^2(M)$ and $\B_2 = \tilde H^2(M_1)$ because $N_g$ is not bounded as an operator between those two spaces. On the other hand, one cannot replace $\B_1$ by $H^1(M)$ because $N_g(H^1(M))$ is not closed in $\tilde H^2(M)$. For more details, and for the needed modifications in the arguments, we refer to \cite{SU-Duke, SU-JAMS}. 
\end{remark}

The following proposition is well known, see e.g., \cite{Taylor-book1}. We define Sobolev spaces on a compact manifold by fixing a finite atlas. We assume in the proposition below that $N$ is of order $-m\le0$ in view of applications to tomography problems. The order $-m$ does not need to be non-positive, and one can study $N$ as an operator mapping $H^s$ into $H^{s+m}$ for $s$ not necessarily  $0$ by the usual technique --- applying powers of $1-\Delta$ to the left and right.

\begin{proposition}  \label{pr2}
Let $N\in \Psi^{-m}(M)$, $m\ge0$ be an elliptic \PDO\ on a compact manifold $M$.  

(a) Then for any $l\ge 0$, one has
\be{28a}
\|f\|_{L^2(M)} \le C_l\left(\|Nf\|_{H^{m}(M)}  + \|f\|_{H^{-l}(M)}    \right).
\ee

(b) Assume in addition, that $N$ is injective on some (closed) subspace $\mathcal{L}\subset L^2(M)$. Then there exists $C>0$, so that
\be{29a}
\|f\|_{L^2(M)} \le C \|Nf\|_{H^m(M)}, \quad \forall f\in \mathcal{L}.
\ee
\end{proposition}

\begin{proof}
Part (a) follows directly from the existence of a parametrix. Indeed, there is $Q\in \Psi^{m}$ so that $QAf =f+Kf$, where $K$ has a smooth kernel. Note that for this proof it is only enough to have a parametrix $Q$ of finite order $l$, i.e., to have $K$ that maps $L^2$ into $H^l$. Then \r{28a} follows directly.
 
Note that the inclusion $L^2\to H^{-l}$ is a compact map. Part (b) follows from Proposition~\ref{pr2F} above. 

Note also that an alternative way to prove (b) is to notice that it  follows directly from (a), by \cite[Proposition~5.3.1]{Taylor-book1}. 
\end{proof}

Recall that $\Psi^{-m}(M)$ is a Fr\'echet space with semi-norms $p_k(A)$, $k=1,2,\dots$ given by the semi-norms (they are actually norms) of its amplitude  in the finitely many charts $(U_j,\chi_j   )$ of $M$. The norms $p_k$ are given by
\[
p_k(a) = \max_j \sup_{x\in \bar U_j,\, \xi\in\R^n} \max_{|\alpha|+|\beta|+|\gamma|\le k}(1+|\xi|)^{m+|\gamma|} |\partial_x^\alpha \partial_y^\beta \partial_\xi^\gamma a_j(x,y,\xi)|,
\]
where 
$a_j$ is the full symbol of $A$ in the $j$-th coordinate chart. 

\begin{proposition}  \label{pr5} There exists $k\gg1$, depending only on $\dim M$ with the following property. Let $A=A_0\in \Psi^{-m}$ satisfy the  assumptions of Proposition~\ref{pr2}(b).  Then \r{29} holds for any $A\in \Psi^{-m}$ close enough to $A_0$ in the norm $p_k$ with a constant $C$ independent of $A$. 
\end{proposition}

\begin{proof}
Note that $\Psi^{-m}$ equipped with the norm $p_k$, $k\ge 2n+m+1$, is a subspace of the space of linear bounded operators mapping $L^2(M)$ into $H^m(M)$, see the appendix in \cite{SU-IBSP}. Now the proposition follows from  Proposition~\ref{pr2F}(b). 
\end{proof}

\begin{remark} Proposition~\ref{pr2} can be applied to compact manifolds $M$ with boundary; then we require  the symbol to be defined and to satisfy the symbol estimates for $x$ in the closed $M$. One can then extend it as an elliptic symbol to some extension of $M$ without boundary, and treat $L^2(M)$ as $\mathcal{L}$. The same remark applies to Proposition~\ref{pr5} because on can choose the extension of the symbol in a way that is continuous in any fixed $p_k$ norm.
\end{remark}

Finally, we formulate a direct corollary of these statements.

\begin{corollary}  \label{cor1}
Let $A: L^2(M)\to \H$,  be a bounded linear operator, where $M$ is a compact manifold  without boundary, and $\H$ is a Hilbert space. 
Assume that $N=A^*A$ is an elliptic \PDO\ of order $-m\le0$. 
Assume also that $A$ is injective on a (closed) subspace $\mathcal{L}$ of $L^2(M)$.

(a)  Then $N : \mathcal{L} \to\mathcal{L}$ is onto, and 
\be{c1}
\|f\|_{L^2(M)} \le C\|Nf\|_{H^m(M)}, \quad \forall f\in \mathcal{L}.
\ee

(b) Estimate \r{c1} remains true with an uniform $C>0$ if $N$ is replaced by any other operator close enough to $N$ in the operator norm $ L^2(M)\to \H$, and in particular for all \PDO s in a small enough neighborhood of $A$ in the norm $p_k$, $k\gg1$. 
\end{corollary}

\begin{proof}
The only statement left to prove is that $N : \mathcal{L} \to\mathcal{L}$ is onto. This follows from the fact that $N$ is injective and self-adjoint, hence it has a trivial cokernel. Next, since $N$ is Fredholm, its range is closed.
\end{proof}

In particular, by Lemma~\ref{lemma1}, Corollary~\ref{cor1}(b) implies that $A$ remains injective under a small perturbation  (which also follows from the fact that $A|_{\mathcal{L}}$ is upper semi-Fredholm). Also, \r{c1} implies an estimate of the kind $\|f\|\le C\|Af\|_*$ with a suitable norm $\|\cdot\|_*$. 

Note that Corollary~\ref{cor1} can be applied to manifolds with boundary as in the example below. 

\begin{example}
We return to Example~\ref{ex2}. Let $\Omega_1\Supset \Omega$ as before. Assume that $N_w= I_w^*I_w$ is an elliptic \PDO\ (of order $-1$) in $\Omega_1$. To apply Corollary~\ref{cor1} directly, one can embed $\Omega_1$ into a compact manifold and extend $N_w$ as an elliptic operator there in a way independent of $w$. Alternatively, on can work with the parametrix in $\Omega_1$ and eventually restrict to $\bar\Omega$. By \r{26}, the principal symbol of $N_2$ is given by (see \cite{SU-Duke, SU-JAMS, FSU} derivation in even more general cases)
\[
\sigma_{\rm p}(N_w)(x,\xi) = 2\pi \int_{S^{n-1}} |w(x,\theta)|^2 \delta(\xi\cdot\theta)\, \d \theta,
\]
where $\delta$ is the Dirac delta function, and clearly, $\sigma_{\rm p}(N_w)$ is homogeneous in $\xi$ of order $-1$. The singularity at $\xi=0$ can be cut-off. Since $w$ is at least continuous, we see that the ellipticity assumption is equivalent to the following:
\be{211}
\forall (x,\zeta)\in \Omega\times S^{n-1}, \; \exists \theta\in S^{n-1}, \;\theta\perp\zeta; \;  \mbox{so that}\; w(x,\theta)\not=0.
\ee
In invariant terms, $(x,\xi)\in S^*\Omega$, while $(x,\theta)\in S\Omega$. Another way to express the same is to say that
\be{212}
N^*(\supp^0 w) \supset T^*\Omega,
\ee
where $\supp^0 w$ is the open set where $w\not=0$. 
Under this condition, using the continuity of $w$ in $S\bar\Omega \cong\bar \Omega\times S^{n-1}$, one can get the same with $\Omega$ replaced by $\Omega_1$. By Proposition~\ref{pr2}, we then get
\be{213}
\|f\|_{L^2(\Omega)}  \le C \|N_w f\|_{H^{1}(\Omega_1)},
\ee
under the assumption that $I_w$ is injective on $L^2(\Omega)$. Since $N_w$ is elliptic, it is enough to assume only that $I_w$ is injective on $C_0^\infty(\Omega)$. Without that assumption, one has \r{28a}; and \r{213} on $(\Ker I_w)^\perp$. 

To analyze the continuous dependence on $w$, let us note first that given $k\ge0$, one can construct a bounded extension operator extending functions in $C^k(\bar\Omega)$ to $C^k(\bar\Omega_1)$. With the aid of this operator, one can map a neighborhood of some $w_0\in C^k(\bar\Omega)$ to some neighborhood of an extended $w_0$ in $C^k(\Omega_1)$. Then we can apply Corollary~\ref{cor1} to get the following.

\begin{theorem} \label{thm2}
Let $w\in C^1(\bar\Omega\times S^{n-1})$ satisfy \r{211}. Assume that $I_w$ is injective on $C^\infty_0(\Omega)$.  
Then $N_w :L^2(\Omega) \to L^2(\Omega)$ is onto, and  there exists a constant $C>0$ so that
\be{214}
\|f\|_{L^2(\Omega)}  \le C\|N_wf\|_{L^2(\Omega_1)}, \quad \forall f\in L^2(\Omega).
\ee
Moreover, \r{214} remains true under small $C^1(\bar\Omega)$ perturbations of $w$. 
\end{theorem}
\end{example}

Without the injectivity assumption, \r{214} holds on $(\Ker I_w|_{C_0^\infty(\Omega)})^\perp$. Theorem~\ref{thm2} is a weaker version of the results in \cite{FSU} and is presented here just to illustrate the method. Note that $w\in C^k$ with $k=1$ is not enough to guarantee  that operator convergence in the $p_k(N)$ norm implies convergence in the norm $\|N\|_{L^2(\Omega) \to H^1(\Omega_1)}$, viewing $N=N_w$ as a \PDO. On the other hand, $N_w$ is also an operator with singular kernel, see \r{26}, and $w\in C^1$ is enough for $L^2\to H^1$ continuity. We refer to \cite{FSU} for details.

We show finally that for an injective \PDO, ellipticity is a necessary and sufficient condition for the estimate \r{c1} in Corollary~\ref{cor1} to hold.

\begin{proposition}  \label{pr6}  \ 

(a) 
Let $N$ be a \PDO\ of order $-m\le0$ on a compact manifold $M$. Then \r{29a} holds for any $f\in L^2(M)$ if and only if $N$ is injective and elliptic.

(b) Let $M$ be a compact manifold with boundary and let $M_1$ be another compact manifold so that $\Mint_1\supset M$. Assume that $N$ is a \PDO\ of order $-m\le0$ defined on some neighborhood of $M$. Then 
\be{c2}
\|f\|_{L^2(M)} \le C\|Nf\|_{H^m(M_1)}, \quad \forall f\in L^2(M)
\ee
if and only if $N$ is injective on $L^2(M)$ and elliptic on $\pi^{-1}(M)$;  more precisely,  
\[
\inf_{x\in M}  \liminf_{|\xi|\to\infty}|\xi|^m |p_0(x,\xi)|>1/C,
\]
where $p_0$ is the principal symbol of $N$.
\end{proposition}

\begin{proof}  

To prove (a), we will work locally near a fixed point on $M$. In fixed local coordinates near $x_0$, let $p(x,\xi)$ be the symbol of $N$, multiplied by a standard cut-off supported near some $x_0$. We will study $p(x,D)$ in $\R^n$ first.  Fix $\xi_0\not=0$ and apply $N$ to the normalized ``coherent state'' 
\[
f_\lambda(x) = (\lambda/\pi)^{n/4} e^{\i\lambda  x\cdot\xi_0-\lambda|x-x_0|^2/2}.
\]
Then
\[
\hat f_\lambda(\lambda\xi) = (\lambda/\pi)^{n/4} (2\pi/\lambda)^{n/2}e^{-\i \lambda(\xi-\xi_0)\cdot x_0  -{\lambda}  |\xi-\xi_0|^2/2}.
\]
Then we get
\[
[p(x,D) f_\lambda] (x) = e^{\i \lambda x_0\cdot\xi_0}    (\lambda/\pi)^{n/4} 
 (2\pi\lambda)^{-n/2}  \int e^{\i \lambda (x- x_0)\cdot\xi- \lambda |\xi-\xi_0|^2/2}   p(x,\lambda\xi)\,  \d\lambda\xi.
\]
For any $C>0$, the contribution to the integral above coming from integrating over $|\xi-\xi_0|>1/C$ is $O(\lambda^{-N})$, $\forall N>0$. To estimate the effect of replacing $p(x,\lambda\eta)$ by $p(x,\lambda\xi_0)$ in the integral above,  notice first that $p(x,\lambda\xi) -p(x,\lambda\xi_0) = \zeta(x,\xi,\xi_0,\lambda)\cdot(\xi-\xi_0)$ with $\zeta=O(1)$, therefore 
\be{pf}
\int_{|\xi-\xi_0|<C}|\xi-\xi_0|e^{-\lambda|\xi-\xi_0|^2/2}  \d\xi =  O(\lambda^{-n})
\ee
(pass to polar coordinates $r\omega=\xi-\xi_0$). 
Thus we get
\[
\begin{split}
[p(x,D)& f_\lambda] (x) \\
&= e^{\i \lambda x_0\cdot\xi_0} (\lambda/2\pi )^{n/2} (\lambda/\pi)^{n/4}
   \left[\int e^{\i \lambda (x-  x_0)\cdot\xi-\lambda |\xi-\xi_0|^2/2}   p(x,\lambda\xi_0)\,  \d\xi +O(\lambda^{-n}) \right] \\
&= e^{\i \lambda x_0\cdot\xi_0}    (\lambda/2\pi)^{n/2}   (\lambda/\pi)^{n/4} \left[ p(x,\lambda\xi_0)
   (2\pi/\lambda)^{n/2} e^{\i\lambda(x-x_0)\cdot\xi_0 - \lambda |x-x_0|^2/2}+ O(\lambda^{-n} )\right]\\
&=  p(x,\lambda\xi_0) f_\lambda(x) + O(\lambda^{-n/4}) .
\end{split}
\]
Assume now that \r{29a} holds. It is enough to consider the case $m=0$. Then, choosing $f=f_\lambda$, we get
\be{ell1}
1/C \le \|p(\cdot,\lambda\xi_0)f_\lambda(\cdot)\| + O(\lambda^{-n/4}).
\ee
The function $\phi_\lambda(x-x_0) = |f_\lambda(x)|^2 = (\lambda/\pi)^{n/2}e^{-\lambda|x-x_0|^2}$ is normalized so that it has $L^1$ norm equal to $1$. Therefore,
\be{ell2}
\|p(\cdot,\lambda\xi_0)f_\lambda(\cdot)\|^2 = \int \phi_\lambda(x-x_0)|p(x,\lambda\xi_0)|^2\, \d x  =|p(x_0,\lambda\xi_0)|^2  +O(\lambda^{-n})
\ee
by the argument already used above, see \r{pf}. Relations \r{ell1} and \r{ell2} show that $p$ is elliptic at $(x_0,\xi_0)$. More precisely, we get
\be{215}
|p(x,\xi)|\ge 1/C - C_1|\xi|^{-n/4},  \quad \forall (x,\xi)\in T^*M,
\ee
where $C>0$ is the constant $C$ in \r{29a}, and $C_1>0$ depends on $p$ only. 

To relate $N$ and $p(x,D)$, note that in a fixed coordinates system near some $x_0$, $R= N-p(x,D)$ is smoothing when acting on distributions supported near $x_0$. Consider then $(N-p(x,D)) \chi f_\lambda$, where $\chi$ is a standard cut-off near $x_0$. Then 
\[
p(x,D)f_\lambda = p(x,D)(1-\chi)f_\lambda+ p(x,D)\chi f_\lambda = p(x,D)(1-\chi)f_\lambda+(N-R)\chi f_\lambda.
\]
Now, $(1-\chi)f_\lambda= O(e^{-\lambda/C})$, $N\chi f_\lambda= O(\lambda^{-\infty})$ as can be seen by integrating by parts and using the smoothness of the kernel of $R$. Therefore, $N\chi f_\lambda = p(x,D) f_\lambda + O(\lambda^{-\infty})$, and this completes the proof of  (a). 

To prove (b), we proceed in the same way in a neighborhood of any $x_0\in \Mint$ but we apply $p(x,D)$ to $\chi f_\lambda$ instead to $f_\lambda$, where $\chi$ is a standard cut-off near $x_0$ supported in $M$. As $x_0$ gets closer to $\bo$, the derivatives of $\chi$ will have larger norms. This will affect the constant $C_1$ in \r{215} that nay not be uniform in $x$. The conclusion (b) is still true, however. Note that if $N$ has a polyhomogeneous symbol, then the sub-principal symbol must be uniformly bounded for $x\in M$,  because we assume that $N$ is a \PDO\ in the larger $M_1$. Then $C_1$ is actually uniformly bounded for $x\in M$.
\end{proof}

\subsection{Non-sharp linear stability estimates} 
In view of the analysis in Section~\ref{sec_IP}, we are also interested whether an estimate of the type \r{13}, weaker than \r{c1} can still be proven when $N$ is not an elliptic \PDO. In other words, what conditions would imply \r{c1} in different $H^s$ spaces? Ellipticity of $N$ is not necessary anymore, and for certain classes of hypoelliptic operators, one can still have \r{c1} with a loss of one derivative, for example.  We are not going to review those classes of operators. Without a proof, we will only mention that if the full symbol of $N$ vanishes on some open conic set, \r{c1} cannot hold in any Sobolev spaces. This is well known and used in tomography to determine subsets of lines/curves on which the X-ray transform can or cannot determine the function we integrate in a stable way. In particular, in  Example~\ref{ex2}, if $T^*\setminus N^*(\supp w)$ contains an open conic set, see \r{212}, then no $H^{s_1}\to H^{s_2}$ stability estimate of the type \r{213} is possible.

\subsection{Injectivity of the differential $A_{f_0}$} In this section so far we were trying to find sufficient conditions that would guarantee a stability estimate under the assumption that the differential $A_{f_0}$ is known to be continuous. It is not the scope of this work to study in detail the typical approaches to prove that the differential is continuous. We will only briefly mention some of them: smallness of the coefficients is often used when the the constant coefficient case is easy to study directly. Another technique is the use of the Analytic Fredholm Theorem \cite{Reed-Simon1} that, when can be applied, gives injectivity of $A_{f}$ for generic $f$'s. When the unknown coefficient is  in the principal symbol, as in the boundary rigidity question and related hyperbolic inverse problems \cite{SU-IMRN}, the calculus of the analytic \PDO s can be applied. 

\section{An example: The inverse back-scattering problem for the acoustic wave equation}
\label{sec_4}

Consider the acoustic wave  equation
\be{e1}
\left(\partial_t^2 - c^2(x)\Delta\right)u=0, \quad t\in\R, x\in \R^3
\ee
with a variable speed $c(x)>0$ that equals $1$ for large $x$, i.e., 
\be{e2}
c(x)=1 \quad \mbox{for $|x|>\rho$}
\ee
with some $\rho>0$. Let $S(s,\omega,\theta)$, where $(s,\omega,\theta)\in \R\times S^2\times S^2)$, be the scattering kernel associated with $c$, see below. In \cite{SU-IBSP} we showed that if $c$ is close enough to $1$ in the $C^9$ norm, then the back-scattering kernel $S(s,-\theta,\theta)$ determines $c$ uniquely.

We recall some facts about the time-dependent scattering theory for \r{e1}, see \cite{SU-IBSP}. For simplicity, we  work in space dimension $n=3$. Given $\theta\in S^2$, we denote by $u(t,x,\theta)$ the {\em distorted plane wave}  defined as  the solution of \r{e1} satisfying
\be{e4}
u|_{t\ll0}=\delta(t-x\cdot\theta).
\ee
Then we set $u_{\rm sc}= u-\delta(t-x\cdot\theta)$. 
In the Lax-Phillips scattering theory \cite{LP}, see also ..., The {\em asymptotic wave profile} $u^\sharp$ of $u$ is defined by 
\[
u_{\rm sc}^\sharp(s,\omega,\theta) = \lim_{t\to\infty}(t+s)\partial_t u_{\rm sc}(t,(t+s)\omega,\theta),
\]
where $s\in \R$, $\omega\in S^2$. 
The {\em scattering kernel} $S(s,\omega,\theta)$ is given by
\[
S(s,\omega,\theta) = -\frac1{2\pi}u_{\rm sc}^\sharp(s,\omega,\theta).
\]
Then $S(s'-s,\omega',\omega)$ is the Schwartz kernel of $\mathcal{R}(\mathcal{S}-\Id)\mathcal{R}^{-1}$, where $\mathcal{S}$ is the scattering operator, and $\mathcal{R}$ is the Lax-Phillips translation representation.

One can formulate this problem with stationary data in mind. The scattering amplitude $a$, that is also defined through the stationary scattering theory, is essentially just a Fourier transform (a stable and invertible operation) of $S$ in the $s$ variable, more precisely,
\be{e3}
\frac{\i\lambda}{2\pi}  a(\lambda,\omega,\theta) = \int e^{-\i s\lambda}S(s,\omega,\theta)\, \d s,
\ee
see \cite{SU-IBSP}. 

Set $\omega=-\theta$. Then 
\be{e5c}
\supp S(\cdot,-\theta,\theta) \subset (-\infty,2\rho]. 
\ee
Let $P_\theta$ be the plane $x\cdot\theta=-\rho$, and let ``$\mbox{dist}$'' be the distance function in the metric $c^{-2}dx^2$.  
Fix $T$ so that
\be{e9}
T>2\max\{\mbox{dist}(P_\theta,x);\; x\in B(0,\rho), \, \theta\in S^2\}-2\rho,
\ee
where $B(0,\rho)$ is the ball with center $0$ and radius $\rho$, 
see also \r{e9'}. 
We use the information contained in $S$ for $s\ge -T$ only. We will denote by $h$  the Heaviside function (the characteristic function of $\R_+$).   Our data now is 
\be{e5a}
S_T(s,-\theta,\theta) := h(s+T)S(s,-\theta,\theta).
\ee
Instead of working with $S_T$, we will work with its stationary analog $a_T$ defined by
\be{e8}
a_T(\lambda,-\theta,\theta) := \frac{2\pi}{\i \lambda^3}\int e^{-\i s\lambda}S_T(s,\omega,\theta)\, \d s,  \quad \lambda>0.
\ee
The extra factor $\lambda^{-2}$ there, compared to \r{e3}, is there for convenience. Clearly, $\lambda^2a_T= \hat S_T* \hat\chi_T$, where the hat indicates Fourier transform w.r.t.\ $s$. Getting a stability estimate in terms of $a_T$ instead of $a$ is even stronger because it uses less information. Since $a_T$ is analytic in $\lambda$, we can throw away any finite interval without losing information. Our next theorem shows that we are not losing stability, either. We fix $\lambda_0\ge0$ and restrict $\lambda$ to the interval $\lambda\ge\lambda_0$. 

\begin{theorem}  \label{thm_e1} There exist $\eps>0$, $k\ge2$, $\mu\in (0,1)$ with the following property. For any two  $\tilde c$, $c$ satisfying \r{e2} and 
\be{et1}
\|c-1\|_{C^k}+\|\tilde c-1\|_{C^k}\le \eps,
\ee
and for any $\lambda_0\ge0$, we have 
\be{et2}
\|\tilde c-c\|_{L^\infty} \le C\left(\sup_{\lambda>\lambda_0, \, \theta\in S^2}  (1+\lambda) \big|(\tilde a_T-a_T)(\lambda,-\theta,\theta)\big|\right)^\mu,
\ee
where $T$ is fixed so that it satisfies \r{e9} w.r.t.\ both $\tilde c$ and $c$.
\end{theorem}

\begin{proof}
We start with standard geometric optics arguments, see also \cite[Prop.~3.1]{SU-IBSP}. Let $h_j(t)=t^j/j!$ for $t\ge0$, and $h(t)=0$ for $t<0$; then $h_j'=h_{j-1}$, $j\ge 0$, with the convention $h_0=h$, $h_{-1}=\delta$. 

\begin{proposition}\label{pr_go}
If $c\in C^\infty$, and $\eps\ll1$, then $\forall N$, 
\be{e6}
u= \sum_{j=-1}^{N}\alpha_j(x,\theta)h_j(t-\phi(x,\theta))  +
r_N(t,x,\theta) 
\ee
for $|x|\le\rho$, 
where $\phi$ solves the eikonal equation
\be{e7}
c^2(x)|\nabla\phi|^2=1, \quad \phi|_{x\cdot\theta< -\rho}=x\cdot\theta;
\ee
$\alpha_j$ solve the corresponding transport equations;   and $r_N\in C^{N}$. 
\end{proposition}

\begin{proof}
The construction of $\alpha_j$ is standard, we will focus our attention on $r_N$. The latter solves
\[
\left(\partial_t^2 - c^2(x)\Delta\right) r_N  = (c^2\Delta \alpha_N) h_{N}( t-\phi(x,\theta)) , \quad r_N|_{t<-\rho}=0.
\]
By standard energy estimates, $r_N\in H^{N+1}_{\rm loc}$, see also \cite{SU-IBSP}. By Sobolev embedding theorems, we get $r_N\in C^{N-1}$. To prove that $r_N\in C^{N}$, we apply this to $r_{N+1}$; then $\alpha_{N+1}h_{N+1}(t-\phi(x,\theta))+r_{N+1}\in C^N$.  
\end{proof}

Note that $\rho+\phi(x,\theta)= \mbox{dist}(P_\theta,x)$. The expansion \r{e6}  is valid in $\overline{B(0,\rho)}$ if the eikonal equation has a smooth solution there. The latter is guaranteed by the following geometrical condition. Let $\gamma_{z,\theta}(t)$ be the geodesic in the metric $c^{-2}dx^2$ issued from $z$ in the direction of $\theta$. Then we want $z\in P_\theta$, and $t$ to be coordinates in $\overline{B(0,\rho)}$. The is certainly true if $c$ is close enough to a constant in $C^2$ that we assume. The leading coefficient $\alpha_{-1}$ solves a homogeneous transport equation, and in particular, $\alpha_{-1}>0$. If $c=1$, then $\alpha_{-1}=1$, $\alpha_j=0$ for $j\ge0$, and $\phi=x\cdot\theta$.

Note that \r{e6} remains true if $c\in C^k$, $k\gg2$ but for some $N=N(k)$, and $N(k)\gg1$ if $k\gg2$. Moreover, $\{\alpha_j\}_{j=-1}^N$ depend continuously on $c\in C^k$.  Therefore, under the assumptions of the theorem, for any fixed $N$,
\be{ea}
\|\alpha_{-1}-1\|_{C^N} \le C\eps, \quad \|\alpha_{j}\|_{C^N} \le C\eps, \quad j=1\dots N, \quad \|r_N\|_{C^N}\le C\eps,
\ee
if $k\gg1$.  The $C^N$ norm of $\alpha_j$ is taken in $\overline{B(0,\rho)}\times S^2$, while the norm of $r_N$ is taken in for $(x,\theta)\in\overline{B(0,\rho)}\times S^2$, and $t$ in any fixed finite interval (under the assumption  $\eps\ll1$). Similarly,
\be{ea1}
\|\phi-x\cdot\theta\|_{C^N}\le C\eps,
\ee
see also \cite{SU-IBSP}.

Let $c$ and $\tilde c$ be two sound speeds, and denote by $\tilde S$, $\tilde u$ the scattering kernel $S$ and the solution $u$ related to $\tilde c$. 
The following formula is proven in \cite{SU-IBSP}:
\be{e5}
(\tilde S-S)(s,\omega,\theta) = \frac1{8\pi^2}\partial_s^3\iint (\tilde c^{-2}-c^{-2}) \tilde u(t,x,\theta) u(-s-t,x,-\omega)\, \d t\, \d x.
\ee
The integral above makes sense even though $u$ and $\tilde u$ are distributions.

By the finite speed of propagation  for \r{e1}, see e.g.\ \cite{Evans}, 
\be{e8a}
\supp u(t,\cdot,\theta)  \subset \left\{ x;\; \phi(x,\theta)\le t \right\}, \quad
\supp u(-s-t,\cdot,\theta)  \subset \left\{ x;\; \phi(x,\theta)\le -s-t \right\}.
\ee
Therefore, for $s\ge -T$, on the support of the integrand in \r{e8a}, we have
\[
t\le T-\phi(x,\theta), \quad -s-t \le T-\phi(x,\theta),
\]
see \r{e8a}. 
This gives an upper bound for $t$ and $-s-t$ for $|x|\le\rho$. We want that upper bound to be so that the singular front of both waves involved in \r{e3}, with $\omega=-\theta$ there, would pass through the whole ball $B(0,\rho)$. Physically, this is equivalent to the requirement that we make our measurements for an interval of time so that the  front of the wave $u(t,x,\theta)$ would have enough time to pass through the whole $B(0,\rho)$, and then return back a scattered signal. This explains the  choice of $T$: we want $T$ to be   so that $T-\phi(x,\theta)> \phi(x,\theta)$, $\forall x\in B(0,\rho)$, i.e., 
\be{e9'}
T>2\max\{\phi(x,\theta); \; |x|\le\rho, \, |\theta|=1\}.
\ee
This is equivalent to \r{e9}. 
Therefore, we study $S$ restricted to $-T\le s\le 2\rho$ (the upper bound is not a restriction actually, see \r{e5c}), and the length of that interval, assuming for a moment equality in \r{e9}, equals  twice the length of the longest  geodesic issued from some support plane $P_\theta =\{x\cdot\theta=-\rho\}$, and staying in $B(0,\rho)$. 

For $x\in B(0,\rho)$, $s\ge -T$, and $t>T+\rho$, we have that $-s-t<-\rho$, and then $u(-s-t,x,\theta)=0$, see \r{e8a}. Similarly, for $x\in B(0,\rho)$, $s\ge -T$, and $-s-t>T+\rho$, we have that $t<-\rho$, and then $\tilde u(t,x,\theta)=0$. This shows that in the integral \r{e5}, the solutions $\tilde u$ and $u$ can be replaced by the same solutions restricted to $t\le T+\rho$, and this will not change the l.h.s.\ when $s>-T$. Let $\chi$ be a smooth function so that 
\be{chi}
\mbox{$\chi(t)=1$ for $-\rho\le t\le T+\rho$, \quad $\chi\in C_0^\infty(\R)$.} 
\ee
Set
\be{e10}
u_T = \chi(t) u(t,x,\theta),
\ee
and we define $\tilde u_T$ in a similar way. Note that the subscript $T$ in $S$ and $u$ has a different meaning. The cut-off for $t<-\rho$  in \r{e10} does not change $u$ for $x\in B(0,\rho)$ because $u=0$ there. 
Then
\be{e11}
(\tilde S_T-S_T)(s,-\theta,\theta) = \frac1{8\pi^2}\partial_s^3\iint (\tilde c^{-2}-c^{-2}) \tilde u_T(t,x,\theta) u_T(-s-t,x,\theta)\, \d t\, \d x.
\ee
Take Fourier transform of both sides to get
\be{e12}
(\tilde a_T-a_T)(\lambda,-\theta,\theta)= -\frac1{4\pi}  \int \left(\tilde c^{-2}(x) -c^{-2}(x)\right ) \tilde v_T(x,\theta,\lambda)v_T(x,\theta,\lambda)\, \d x,
\ee
where 
\be{e13}
v_T(x,\theta,\lambda) = \int e^{\i \lambda t}u_T(t,x,\theta)\, \d t.
\ee

\begin{remark}
Note that if we replace $u_T$ by $u$ on the r.h.s.\ above, we get the ``distorted harmonic wave'' solution $v=e^{\i\lambda x\cdot \theta} + v_{\rm sc}$, where $v_{\rm sc}$ is outgoing, i.e., it satisfies the Sommerfeld radiation conditions. In our case, $v_T$ is a convolution of $v$ with the Fourier transform of a certain characteristic function. That convolution does not affect the high $\lambda$ asymptotic of $v$, so it is not surprising that below we get the  expansion familiar from the stationary theory. The advantage that we have here is that we do not need to estimate the contribution of large $t$'s, and instead of using estimates of the cut-off resolvent for non-trapping systems,   we can just use the geometric optics expansion \r{e6} that is more direct,  easier to justify for $c$'s of finite smoothness, and easier to perturb near a fixed $c$. Furthermore, we get an even stronger result by discarding information ($t\gg0)$ that is not needed. 
\end{remark}

Linearizing \r{e12}, we get the following.

\begin{proposition}  \label{pr_e1}
Let $c$, $\tilde c$ be two speeds satisfying the assumptions of Theorem~\ref{thm_e1}.  Then 
\be{e13a}
(\tilde a_T-a_T)(\lambda,-\theta,\theta)= \int \left(\tilde c^{-2}(x) -c^{-2}(x)\right )  v_T^2(x,\theta,\lambda) \, \d x +R(\theta,\lambda)
\ee
where 
\be{el3b}
\sup_{\theta\in S^2,\,\lambda>0}\; (1+\lambda)|R(\theta,\lambda)| \le C \big\|\tilde c^{-2}-c^{-2}\big\|^2_{L^\infty}.
\ee
\end{proposition}

\begin{proof}
The difference $\tilde u-u$ solves
\[
(\partial_t^2- c^2\Delta)(\tilde u-u) = (\tilde c^2-c^2)\Delta \tilde u, \quad (\tilde u-u)|_{t\ll0}=0.
\]
By standard energy estimates, in any compact set,
\[
\|\tilde u_t-u_t\|_{L_x^2}      +   \|\tilde u-u\|_{H^1_x}
\le C \big\|\tilde c^2-c^2\big\|_{L^\infty}.
\]
The constant $C$ depends on $\tilde c$ and $c$ and  is easy to see that it remains uniformly bounded, if $c$ is bounded in $C^1$ that follows from the hypothesis  \r{et1}. Then we get by \r{e13},
\[
(1+\lambda) \|\tilde v_T-v_T\|_{L^2}\le C \big\|\tilde c^2-c^2\big\|_{L^\infty}.
\]
Compare \r{e12} and \r{e13a} to get
\[
\sup_\lambda|(1+\lambda)R(\theta,\lambda)| \le C \big\|\tilde c^{-2}-c^{-2}\big\|^2_{L^\infty},\quad\forall\theta,
\]
where the $L^2$ norm is taken in any compact. 
\end{proof}

We focus our attention now on the linear operator
\be{e14}
Af (\lambda,\theta) = \int v_T^2(x,\theta,\lambda) f(x)\, \d x, \quad f\in L^2(B(0,\rho)).
\ee
The geometric optics expansion \r{e6} yields
\be{e15}
v_T^2(x,\theta,\lambda) = e^{\i 2\lambda\phi(x,\theta)} \big( a_{-1}(x,\theta,\lambda) +  a_0(x,\theta,\lambda)  +... a_N(x,\theta,\lambda)+ R_N(x,\theta,\lambda) \big)^2, \quad \lambda\to\infty,
\ee
where 
\[
a_j = \int e^{\i \lambda t}\alpha_j(x,\theta)\chi(t)h_j(t-\phi(x,\theta))\,\d t, \quad R_N = \int e^{\i \lambda t}\chi(t)  r_N(t,x,\theta)\,\d t.
\]
Set $\eta=2\lambda\theta$, and extend $\phi$ by homogeneity (of order $1$) w.r.t.\ to $\theta$; then $2\lambda\phi(x,\theta)= \phi(x,\eta)$.  The leading term is 
\[
a_{-1}^2 = \chi^2(\phi(x,\theta))\alpha_{-1}^2. 
\] 
We have that $\phi(x,\theta)$ is $O(\eps)$ close to $x\cdot\eta$,  see \r{ea1}.   The natural choice for the space containing $\Ran A$ is therefore $L^2(\R^n_\eta)$, that we will restrict to $|\eta|>2\lambda_0$. In terms of the original variables, this space is isomorphic to $L^2\big((\lambda_0,\infty) \times S^2;\;\lambda^2\d\lambda\d\theta\big)$. Clearly, if $c\in C^\infty$,   \r{e15} implies that the expression in the parentheses on the r.h.s.\ is a classical elliptic symbol. 
For $c\in C^k$, $k\gg1$ the symbol satisfies the symbol estimates up to a finite order $m$ only, with $m\gg1$, when $k\gg1$. To prove this, we first show that $a_j$ is a symbol of order $-j$ (of finite regularity). The properties of $r_N$ that we established do not prove that  $R_N$  satisfies the symbol estimate $|\lambda^j\partial_\lambda^j \partial_{x,\theta}^\alpha  r_N|\le C\lambda^{-N} $ but it is clear that it satisfies $| \partial_\lambda^j \partial_{x,\theta}^\alpha  r_N|\le C\lambda^{-N} $. We can replace $N$ by $N+m$ now and this argument is enough to show that $A^*A$ is a \PDO\ with symbol of finite regularity $m$. This is summarized in the following.

\begin{proposition}  \label{pr_ee}
Let $c\in C^k(\R^n)$ satisfy \r{e2}. Then
\[
v_T\big(x,\eta/|\eta|,|\eta|/2\big) = e^{\i \lambda\phi(x,\eta)} a(x,\eta) 
\]
where $\phi$  and  $a$ belong to $C^N\big(\overline{B(0,\rho)}\times S^2\big)$ with $N=N(k)\to\infty$, as $k\to\infty$; $a(x,\eta)$  is a formal elliptic symbol of order $0$ satisfying the symbol estimates with all its derivatives up to order $N$.
\end{proposition}

We now consider $A$ as the operator given by \r{e14} in the following spaces
\be{e16}
A : \B_1'  :=  L^2(B(0,\rho))  \longrightarrow B_2' := L^2\big((\lambda_0,\infty)\times S^2;\;\lambda^2\d\lambda\d\theta\big).
\ee
Then 
\be{e18}
A^*Af(x) = \int_{S^2}\int_{\lambda>\lambda_0} \int_{|x|\le\rho} e^{-\i\lambda 2(\phi(x,\theta) -\phi(y,\theta) )} \bar a(x,\theta,\lambda) a(y,\theta,\lambda) f(y) \, \d y\,\lambda^2 \, \d\lambda\, \d\theta.
\ee
By a standard argument, define $\xi=\xi(x,y,\eta)$ by $2\lambda \phi(x,\theta) -\phi(y,\theta)=(x-y)\cdot\xi$, where $\eta=2\lambda\theta$ as above. The latter equation certainly has a solution for $|x|\le\rho$ if $c$ is close enough to a constant in $C^k$, $k\gg1$, and that solution is close enough to $\xi=\eta$ in $C^l$ with $l\gg1$ when $k\gg1$. Then we make a change of variables from $\eta$ to $\xi$ to get that $A^*A$ is a \PDO\ of order $0$ in a neighborhood of $B(0,\rho)$, and it is clearly elliptic. 

To apply the results in section~\ref{sec_2}, and especially Proposition~\ref{pr2}(b), we have to show that $A^*A$ is injective. Set $M = B(0,\rho_1)$, where $\rho_1>\rho$ is close enough to $\rho$ so that the construction above remains valid on $M$. Then set $\mathcal{L}=L^2(B(0,\rho))$. Assume first that $c=1$. Then 
\[
A^*Af(x) = \int_{S^2}\int_{\lambda>\lambda_0} \int_{|x|\le\rho} e^{-\i 2\lambda (x-y)\cdot\theta} f(x)  \, \d y\,\lambda^2 \, \d\lambda\, \d\theta
= \pi^3h(|D|-2\lambda_0)f,
\]
where $f$ is extended as zero outside $B(0,\rho)$. It is easy to see that $h(|D|-2\lambda_0)$ is a Fredholm  injective operator on $L^2(B(0,\rho))$, see \cite{S-CPDE} for any $\lambda_0\ge0$ (and it is the identity of $\lambda_0=0$). Therefore, the conditions of Proposition~\ref{pr2}(b) are satisfied for $c=1$ and we get
\be{e19}
\|f\|_{L^2(B(0,\rho))}/C  \le \|A^*Af\|_{L^2(B(0,R))}  \le C\|f\|_{L^2(B(0,\rho))}.
\ee
with some $R>\rho$. This show in particular that $A :\B_1'\to \B_2'$ is bounded, depending continuously on $c\in C^k$, $k\gg1$.  That property is preserved if we increase $\rho$ a bit. Therefore, $A^*: \B_2'\to L^2(B(0,R))$ is bounded a well, and also depends continuously on $c\in C^k$, $k\gg1$. Thus
\be{e20}
\|f\|_{L^2(B(0,\rho))} \le C \|Af\|_{\B_2'} .
\ee
This is the basic estimate that we need to apply Theorem~\ref{thm1}.   The spaces $\B_1$ and $\B_2$ are determined by the estimate in Proposition~\ref{pr_e1}; set $\B_1=L^\infty(B(0,\rho))$, and let $\B_2$ be the Banach  space determined by the norm on the l.h.s.\ of \r{el3b}. We need to show that the interpolation estimates \r{int} hold with a proper choice of $\B_{1,2}''$. Since
\[
\|f\|_{\B_1} = \|f\|_{L^\infty(B(0,\rho))} \le C_\eps\|f\|_{H^{n/2+\eps}},
\] 
$\forall\eps>0$, we use \r{inter} to deduce that the second interpolation estimate in \r{int} holds with $\B_1'' = H^{s_2}(B(0,\rho)) $,     where $n/2+\eps= s_2(1-\mu_1)$. Note that we can make $\mu_1$ as close to $1$ as we want by choosing $s_2$ large enough. Next, we need to choose $\B_2''$ so that the first interpolation estimate in \r{int} holds:
\be{e21}
\|u\|_{\B_2'}\le C\|u\|_{\B_2}^{\mu_2} \|u\|^{1-\mu_2}_{\B_2''}.
\ee
In this case, we can just choose $\B_2''=\B_2$, and $\mu_2=1$ because it is easy to see that
\be{e21x}
\|u\|_{\B_2'}\le C\|u\|_{\B_2}  .
\ee
The condition $\mu_1\mu_2>1/2$ is therefore satisfied when $s_2>n+2\eps$. Then we can choose $\mu=\mu_1$ to be any number in $(0,1)$ so that $1-\mu< n/(2s_2)$. 
The compactness assumption $\|f\|_{\B_1''}\le K$, see \r{16}, is then satisfied if $\|f\|_{H^{s_2}}\le K$, which is guaranteed by \r{et1}, if $k\ge s_2$. 

This completes the proof of Theorem~\ref{thm_e1}.
\end{proof}


\end{document}